\documentclass[11pt]{amsart}
\usepackage[utf8x]{inputenc}
\usepackage[a4paper]{geometry}
\usepackage{amsthm,amstext,amsmath,amscd,amssymb,latexsym}
\usepackage{arydshln}
\usepackage{hyperref}
\usepackage[all,cmtip]{xy}
\usepackage{verbatim}
                 \hypersetup{ pdfborder={0 0 0}, 
                              colorlinks=true, 
                              linktoc=page,
                              pdfauthor={M. C. Brambilla and G. Staglian{\`o}}, 
                              pdftitle={Algebraic boundaries among typical ranks}} 
\usepackage{color}
\usepackage[dvipsnames]{xcolor}
\usepackage{adjustbox}

\newcommand{\CC}{{\mathbb C}}
\newcommand{\PP}{{\mathbb P}}
\newcommand{\GG}{{\mathbb G}}

\newcommand{\KK}{{\mathbb K}}
\newcommand{\RR}{{\mathbb R}}

\newcommand{\Sing}{{\rm{Sing}}}
\newcommand{\rk}{{\rm{rk}}}

\DeclareMathOperator{\CH}{CH}
\DeclareMathOperator{\Gr}{\mathbb{G}}
\DeclareMathOperator{\codim}{codim}

\newtheorem{theorem}{Theorem}[section]
\newtheorem{lemma}[theorem]{Lemma}
\newtheorem{proposition}[theorem]{Proposition}
\newtheorem{corollary}[theorem]{Corollary}
\newtheorem{conjecture}[theorem]{Conjecture}

\theoremstyle{definition}
\newtheorem{definition}[theorem]{Definition}

\newtheorem{example}[theorem]{Example}

\theoremstyle{remark}
\newtheorem{remark}[theorem]{Remark}

\numberwithin{equation}{section}

\title[Algebraic boundaries among typical ranks]{Algebraic boundaries among typical ranks for real binary forms of arbitrary degree}
\author[M. C. Brambilla]{Maria Chiara Brambilla}
\email{{\tt brambilla@dipmat.univpm.it}}
\address{Universit\`a Politecnica delle Marche,   via Brecce Bianche, I-60131 Ancona, Italy}

\author[G. Staglian\`o]{Giovanni Staglian\`o}
\email{{\tt giovannistagliano@gmail.com}}
\address{Universit\`a degli Studi di Catania, Viale A. Doria 5, 95125 Catania, Italy}

\thanks{The first named author is partially supported by MIUR and INDAM}



\begin{document}

\begin{abstract}
We show that the algebraic boundaries of the regions of real binary forms with fixed typical rank are always unions of dual varieties 
to suitable coincident root loci.
\end{abstract}

\maketitle

\section*{Introduction}

Let $f\in R_d=\KK[x,y]_d$ be a binary form of degree $d$, {where $\KK=\RR$ or $\CC$.}
By definition, see e.g.\ \cite{Landsberg}, the \emph{$\KK$-rank} of $f$ is
the minimum integer $r$ such that $f$ admits a decomposition $f=\sum_{i=1}^r \alpha_i(\ell_i)^d$, 
where $\alpha_i\in \KK$ and $\ell_i\in {{\KK}[x,y]}_1$ for $i=1,\ldots,r$. 

The $\CC$-rank of a form, also called {\it complex Waring rank}, has been widely studied by many authors. The case of binary forms was considered and completely solved by Sylvester \cite{Sylv1851}, who proved that the \emph{generic rank}, \emph{i.e.}, 
the complex rank of a general complex binary form of degree $d$,
is $\lceil\frac {d+1}2\rceil$ (see also \cite{Comas-Seiguer}).   
The generic complex rank of forms in more variables is described by the celebrated Alexander-Hirschowitz theorem \cite{AH95} (see also \cite{AH-BO}). 

On the other hand, the {\it real} Waring rank has been studied only in recent years and most of the questions are still open. 
Clearly the real case is particularly relevant for the applications.  
In fact, the notion of tensor rank, which generalizes the Waring rank, has recently attracted great interest 
 in applied mathematics, 
 chemometrics,
 complexity theory, 
 signal processing,
 quantum information theory, machine  learning, and  other current fields of research; see e.g.\ \cite{Kolda2009,Landsberg,Hackbusch2012,Smilde2004,Comon2008,LanOtt,COMON199693,Qi2007}.

When we work on the real field, the notion of generic rank is replaced by the notion of \emph{ typical ranks}. 
A rank is called typical for real binary forms of degree $d$ if {it occurs in an open subset of $R_d$},
with respect to the Euclidean topology.
More precisely, denoting by 
$\mathcal{R}_{d,r} 
$ the interior of 
the semi-algebraic set 
$\{f\in R_d: \rk_{\mathbb{R}}(f)=r\}$ in the real vector space $R_d$, then
a 
rank $r$ is typical exactly when $\mathcal{R}_{d,r}$ is not empty. 
By \cite{Blekherman} it is known  that
a rank $r$ is typical for forms of degree $d$  if and only if $\frac {d+1}2\le r \le d$.

Let us now assume  $\frac {d+1}2 \le r \le d$. 
Following \cite{LeeSturmfels} and \cite{BS18}, we define the \emph{topological boundary} 
$\partial(\mathcal{R}_{d,r})$ as the set-theoretic difference 
of the closure of $\mathcal{R}_{d,r}$
and the interior of the closure of $\mathcal{R}_{d,r}$. Thus, if $f\in\partial(\mathcal{R}_{d,r})$
then every neighborhood of $f$ contains a generic form of real rank equal to $r$ 
and also a generic form of real rank different from $r$. 
We have that
 $\partial(\mathcal{R}_{d,r})$
 is a semi-algebraic subset of $R_d$ of pure codimension one. 
We define the \emph{algebraic boundary} $\partial_{\textrm{alg}}(\mathcal{R}_{d,r})$, also called \emph{real rank boundary}, 
as the Zariski closure 
of the topological boundary  $\partial(\mathcal{R}_{d,r})$
in the complex projective space $\PP(\CC[x,y]_d)$.

The algebraic boundary for maximum rank $r=d$
coincides with the discriminant hypersurface.
Indeed
by \cite{ComonOttaviani} and \cite{CausaRe}, we know that 
the open set
$\mathcal R_{d,d}$ corresponds to the locus of {\it real-rooted} forms, that is forms with all distinct and real roots.
In the opposite case, the algebraic boundary for minimum rank $\overline{r}=\lceil\frac {d+1}2\rceil$
has been described in \cite{LeeSturmfels}. It is irreducible when $d$ is odd, and 
it has two irreducible components when $d$ is even.
From these general results, it follows a complete description of all the algebraic boundaries with low degree $d\le 6$. 

In \cite{BS18}, we completely described the algebraic boundaries for the next two cases, $d=7$ and $d=8$. More precisely, we show in \cite{BS18} that all the boundaries between two typical ranks are unions of dual varieties to suitable {\it coincident root loci}. 
Coincident root loci are well-studied varieties which parametrize 
binary forms with multiple roots, see Section \ref{secCRL} for the precise definitions. 

In this paper we study the algebraic boundaries for forms of arbitrary degree, and our main result is the following:

\begin{theorem}[Theorems~\ref{teoremaOdd} and \ref{casopari}]\label{teoremaIntroduzione}
For any degree $d$ and any typical rank $\frac{d+1}2\le  r \le d$, the algebraic boundary $\partial_{\textrm{alg}}(\mathcal{R}_{d,r})$ is a union of dual
{varieties to}
coincident root loci.
\end{theorem}

Finally, we remark
 that the study of algebraic boundaries 
 for forms with more than two variables is a challenging and quite open problem, see \cite{Michalek2017,Ventura}.

The paper is organized as follows: in the preliminary Section \ref{preliminary}
we recall some basic notions and results about coincident root loci, 
higher associated subvarieties and apolar maps.
Section \ref{pullbacks} is devoted to the detailed analysis of the
pullbacks, via apolar maps, of higher associated varieties to coincident root loci. The main result of this section is Theorem \ref{lemmaEstivo}, whose two corollaries (Corollaries \ref{soluzione-del-problemino} and \ref{pippo}) are key tools in the proof of Theorem \ref{teoremaIntroduzione}.
In Section \ref{alg-boun} we prove Theorem \ref{teoremaIntroduzione}: more precisely, we consider the case of odd degree in Theorem~\ref{teoremaOdd}, and the case of even degree in
Theorem \ref{casopari}.

\section{Preliminary} \label{preliminary}

\subsection{Coincident root loci}\label{secCRL}

Let $r$ be a positive integer.
A partition of $r$ 
is an equivalence class, under reordering, of lists of positive integers
$\lambda=[\lambda_1,\ldots,\lambda_n]$
  such that $\sum_{i=1}^n\lambda_i=r$.
We denote by $|\lambda|$ the length $n$ of the partition.
Alternatively, the partition $\lambda$ can be  represented by the list of integers $m_1,\ldots, m_k$ 
defined as $m_j=|\{i:\lambda_i=j\}|$, and clearly $\sum_{j=1}^k jm_j=r$.

Given a partition $\lambda$ as above,
the {\it coincident root locus}
$\Delta_\lambda\subset \PP^r = \PP(\CC[x,y]_r)$ associated with $\lambda$ is the
set of 
binary forms $f$ of degree $r$ 
which admit a factorization $f=\prod_{i=1}^n \ell_i^{\lambda_i}$
for some linear forms 
$\ell_1,\ldots,\ell_n\in \CC[x,y]_1$.
These varieties 
have been extensively studied, see e.g. \cite{Weyman-1,Katz,Chipalkatti-1,Chipalkatti-2,Kurmann}.

We have a unirational parameterization of degree $m_1!m_2!\cdots m_k!$:
  \[
   \underbrace{\PP^1 \times \cdots \times \PP^1}_{n\ \mathrm{times}} 
   \longrightarrow \Delta_{\lambda}\subseteq \PP^r ,
   \quad
   (\ell_1,\ldots,\ell_n)\mapsto \prod_{i=0}^n \ell_i^{\lambda_i} .
  \]
In particular, the dimension of
 $ \Delta_{\lambda}$ is $n$.
The degree of $ \Delta_{\lambda}$ was determined by Hilbert \cite{Hilbert}. He showed that
\begin{equation*}
\deg(\Delta_\lambda)=\frac{n!}{m_1!m_2!\cdots m_k!}\lambda_1\lambda_2\cdots\lambda_n  .
\end{equation*}

 If $\lambda$ and $\mu$ are two partitions of $r$, we have 
 $\Delta_\mu \subseteq \Delta_\lambda$ if and only if $\lambda$ is a refinement of $\mu$ 
 (equiv., $\mu$ is a coarsening of $\lambda$).
 In \cite{Chipalkatti-1} and subsequently  in \cite{Kurmann} it has been shown that
 the singular locus $\Sing(\Delta_\lambda)$  
is given by the union of $\Delta_{\mu}$ for some suitable coarsenings $\mu$ of $\lambda$ 
(see \cite[Definition 5.2]{Chipalkatti-1} and
\cite[Proposition 2.1]{Kurmann}  for the precise description).
In particular, one has that $\Delta_\lambda$ is smooth if and only if $\lambda_1=\cdots=\lambda_n$.
Otherwise the singular locus is of 
codimension $1$. 

The dual variety $(\Delta_\lambda)^\vee$ of $\Delta_{\lambda}\subset\PP(\CC[x,y]_r)$ is a subvariety of 
the projective space $\PP(\CC[\partial_x,\partial_y]_r)$ of codimension $m_1+1$ (see \cite[Corollary~7.3]{Katz}).
In particular,  
$\Delta_\lambda^\vee$ is a hypersurface 
if and only if $\lambda_i\ge2$ for all $i$. In this case its degree has been computed in \cite[Theorem~1.4]{Oeding} and it is
\begin{equation}\label{luke}
\deg(\Delta_\lambda^\vee) = \frac{(n+1)!}{m_2!\cdots m_k!}(\lambda_1-1)(\lambda_2-1)\cdots(\lambda_n-1)  .
\end{equation}
Moreover, it is shown in 
 \cite[Corollary~2.6]{LeeSturmfels}  that
$(\Delta_\lambda)^\vee\subset \PP(\CC[\partial_x,\partial_y]_r)$  is 
given by the join of the $n-m_1$ coincident root loci 
$\Delta_{(d-\lambda_i+2,1^{\lambda_i-2})}$ for $1\le i\le n$ with $\lambda_i\ge2$.

\subsection{Higher associated subvarieties}\label{HAS-section}
Let $\GG(l,r)=\GG(l,\PP^r)$ denote the Grassmannian of $l$-dimensional projective subspaces of $\PP^r$.
Let $ X\subset \PP^r_{\mathbb C} $ be an irreducible projective variety of dimension $ n $.
Recall that the {\it $j$-th higher associated variety} 
$\CH_j(X)$ to $X$ is the closure of the set of the $(r-n-1+j)$-dimensional subspaces $H\subset \PP^r$ 
such that $H\cap X\neq\emptyset$ and $\dim(H\cap T_xX)\ge j$ for some smooth point $x\in H\cap X$, see \cite[Chapter~3, Section~2~E]{GKZ}. 
As it is known,
 $\CH_j(X)$ has codimension one in $\GG(r-n-1+j,\PP^r)$
 if and only if $0\leq j\leq \dim(X)-\mathrm{def}(X)$, where 
 $\mathrm{def}(X)$ denotes the dual defect of $X$ (see \cite[Chapter~3, Section~2~E]{GKZ} and \cite[Corollary~6]{Kohn}). 
 Recall also that,
 if $j$ is an integer with $0\leq j\leq n$, then
 the {\it $j$-\emph{th} polar degree of $X$}, denoted by $\delta_j(X)$, is 
 the degree of $\CH_j(X)$ in $\mathbb{G}(r-n-1+j,\PP^r)$ if $\CH_j(X)$ is a hypersurface,
 while it is $0$ otherwise (see \emph{e.g.} \cite{Holme1988}, see also \cite[Section~4]{Kohn}). 

For our purpose it is useful to consider a natural generalization of higher associated subvarieties, which we now introduce.
For any integers $j,l$ with $0\leq j\leq l \leq r$ and $j\leq n$, we define 
\[
\mathbb{I}_{j}^{l} = \mathbb{I}_{j}^{l}(X) = {\{(x,B,P)\in (X\setminus\mathrm{Sing(X)}) \times \GG(j,\PP^r)\times\GG(l,\PP^r) : x\in B, B\subseteq T_x X, B\subseteq P\}}.
\]
\begin{proposition}\label{prop0} The scheme $\mathbb{I}_{j}^{l}$ is smooth and irreducible of dimension
 \[
\dim \mathbb{I}_{j}^{l} = -j^{2}+(n-r+l) j+r l-l^{2}+n ,
\]
that is 
\[
 \dim \GG(l,r) - \dim \mathbb{I}_{j}^{l} =  1 + (j+1) (r-n-1+j-l).
\]
\end{proposition}

\begin{proof}
We have the following diagram of natural projections:
\begin{equation}\label{diag}
\xymatrix{ && \mathbb{I}_{j}^{l} \ar@{->}[rd]^{\pi_2} \ar@{->}[ld]_{\pi_1} \\
 & \mathbb{I}_{j}^{j} \ar@{->}[ld]_{q_1} \ar@{->}[rd]^{q_2} & & \GG(l,\PP^r)\\ 
X\setminus \mathrm{Sing}(X) & & \GG(j,\PP^r)  } 
\end{equation}
For each $x\in X\setminus \mathrm{Sing}(X)$, we have
$q_1^{-1}(x)\simeq\{B\in\GG(j,T_x X):x\in B\}\simeq \GG(j-1,n-1)$. 
Thus, the projection $q_1$ 
is a flat morphism with smooth fibers, and therefore 
$\mathbb{I}_j^{j}$ is smooth and irreducible of dimension 
\[
 \dim \mathbb{I}_j^j = n + \dim \GG(j-1,n-1) .
\]
If $(x,B)\in \mathbb{I}_j^j$, then $\pi_1^{-1}(x,B) \simeq  \{P\in \GG(l,\PP^r): B\subseteq P\}\simeq \GG(l-j-1,r-j-1)$. Thus also the projection $\pi_1$ is a flat morphism with smooth fibers, and it follows that  $\mathbb{I}_j^l$ is smooth and irreducible of dimension 
\[
 \dim \mathbb{I}_j^l = \dim \mathbb{I}_j^j + \dim \GG(l-j-1,r-j-1) .
 \]
Hence the claim of the proposition follows.
\end{proof}
\begin{definition}\label{defVarCH}
Let $\overline{\mathbb{I}_{j}^{l}} = \overline{\mathbb{I}_{j}^{l}(X)}$ be the closure of $\mathbb{I}_{j}^{l}$
inside $X \times \GG(j,\PP^r)\times\GG(l,\PP^r)$, and denote by $\pi_2:\overline{\mathbb{I}_{j}^{l}}\to \GG(l,\PP^r)$ the last projection.
The scheme $\pi_2(\overline{\mathbb{I}_{j}^{l}}) = \overline{\pi_2(\mathbb{I}_{j}^{l})} \subset \GG(l,\PP^r)$ will 
be denoted by $\Xi_{j}^{l}=\Xi_{j}^{l}(X)$.
\end{definition}
\begin{remark}
The equations of the closure  $\overline{\mathbb{I}_{j}^{l}(X)}$ inside 
$X \times \GG(j,\PP^r)\times\GG(l,\PP^r)\subset\PP^r\times \PP^{\binom{r+1}{j+1}-1} \times \PP^{\binom{r+1}{l+1}-1} $,
and hence those of 
$\Xi_{j}^{l}(X)\subset\PP^r$
can be explicitly determined by standard elimination techniques, once one knows the equations of $X\subset\PP^r$.
In particular,
we point out that, if 
$I$ denotes the homogeneous ideal of $X\subset\PP^r$, then 
the homogeneous ideal of 
$\Xi_{j}^{l}(X)\subset \mathbb{G}(l,\PP^r)$ can be calculated using the command 
\texttt{tangentialChowForm(I,j,l)}, provided by the package \emph{Resultants} \cite{Resultants-package} included with {\sc Macaulay2} \cite{macaulay2}.
\end{remark}
\begin{remark}\label{codimAspettataCH0}
 It follows from the definition and Proposition~\ref{prop0} that we have
 \begin{equation}
  \codim_{\GG(l,r)}(\Xi_{j}^{l}(X)) \geq 1 + (j+1) (r-n-1+j-l).
 \end{equation}
 When $j=0$ and $l<r-n$,
 we have that $\pi_2$ is birational onto its image. Indeed, 
 if $l$ is less than the codimension of $X$, then a generic $l$-dimensional 
 projective subspace intersecting $X$ meets $X$ at only one point 
 (see also the proof of \cite[Chapter~3, Proposition~2.2]{GKZ}).
 Hence, when $j=0$ and $l<r-n$,
 the above inequality is an equality, that is 
 \begin{equation}\label{codimCH0gen}
  \codim_{\GG(l,r)}(\Xi_{0}^{l}(X)) = r-n-l = \codim_{\PP^r}(X) - l.
 \end{equation}
\end{remark}
\begin{example}\label{eseCHj}
  If $l = r-n-1+j$, then $\Xi_j^l(X)=\CH_j(X)$ is the
  $j$-th associated variety to $X$.
  In particular, we point out that
 $\Xi_0^{r-n-1}(X)$
  is the Chow hypersurface;
  if $\deg(X)\geq 2$,
  then $\Xi_1^{r-n}(X)$ 
 is the Hurwitz hypersurface (see \cite{Sturmfels-Hurw});
 and
  $\Xi_n^{r-1}(X)\subset \mathbb{G}(r-1,\PP^r)$ 
 is the dual variety of $X$.
 \end{example}
 
 \begin{example}
 Let $\mathcal P$ be a complete flag in $\PP^r$, that is, a nested sequence of projective subspaces
 $\emptyset\subset P_0 \subset \cdots \subset P_{r-1}\subset P_{r}=\PP^r$  
 with $\dim P_i = i$, and let $a=(a_0,\ldots,a_{l})$ be 
  a sequence of integers with $r-l\geq a_0\geq \cdots \geq a_l\geq 0$. 
  Then the so-called Schubert variety $\Sigma_a(\mathcal P)\subset \GG(l,r)$ (see \emph{e.g.} \cite[Chapter~1, Section~5]{griffiths-harris})
  coincides with the following intersection:
    \[
   \Sigma_a(\mathcal P) = \bigcap_{i=0}^{l} \Xi_{i}^{l}(P_{r-l+i-a_{i}}) .
  \]
 \end{example}

 \subsection{Apolar maps and apolarity} 

Let $\KK\subseteq\CC$ be a field. 
Let $R=\KK[x,y]$ be a polynomial ring and let $D=\KK[\partial_x,\partial_y]$ be
the dual ring of differential operators.
The ring $D$ acts on $R$ with the usual rules of differentiations, and we have the pairing
\[
 R_d \otimes D_r \to R_{d-r} .
\]

If $f = \sum_{i=0}^d \binom{d}{i} a_i x^{d-i} y^i\in R_d$, the \emph{apolar ideal} $f^{\perp}\subset D$ is given 
by all the operators which annihilate $f$, that is 
\[
 f^{\perp} = \{g(\partial_x,\partial_y)\in D : g(f)=0\} .
\]
For instance,
 if $l=ax+by\in R_1$, then $l^{\perp}$ is generated by the operator $-b\partial_x + a\partial_y\in D_1$.

The component $(f^{\perp})_r = f^{\perp}\cap D_r$ 
is the kernel 
of the linear map 
$
 D_r\to R_{d-r}
$,
which in
the standard basis 
is represented  by the \emph{catalecticant (or Hankel) matrix} (up to multiplying the rows by scalars), see e.g. \cite{EHRENBORG1993157}:
\[
A^{d,r} = \begin{pmatrix}
a_0 & a_1 & \cdots & a_{r}\\
a_1 & a_2 & \cdots & a_{r+1}   \\
\vdots   & \vdots   & \ddots &   \vdots           \\
a_{d-r}   & a_{d-r+1}   & \cdots     &  a_{d}         
\end{pmatrix} .
\]

For a general form $f\in R_d$, with $d \geq r\geq d-r$,
the matrix $A^{d,r}$ has maximal rank, and hence 
$\dim \mathrm{ker}(A^{d,r}) = 2r-d$.
Thus we have a rational map, called the \emph{apolar map},
\begin{equation}
\varPsi_{d,r}: \PP^d\dashrightarrow \Gr(d-r,r)\simeq\Gr(2r-d-1,r)
\end{equation}
which 
associates to a general binary form $f$ of degree $d$ 
the projective
$(2r-d-1)$-dimensional 
 subspace 
 $
 \PP((f^\perp)_r)\subset\PP(D_r)$. 
In coordinates the map $\varPsi_{d,r}$ is defined by the maximal minors of the matrix $A^{d,r}$, thus by forms of degree $d-r+1$.
For $d=r$, the map $\varPsi_{d,d}$ gives an identification between $\PP^d=\PP(R_d)$
and $\Gr(d-1,d)=\PP(D_d)^{\vee}$.

 Whenever $r\geq \frac{d+2}{2}$, we have that
 $\varPsi_{d,r}: \PP^d\dashrightarrow \Gr(2r-d-1,r)$ is a birational map onto its image. This implies 
 that we can recover a general binary form of degree $d$ from the component $(f^\perp)_r\subset f^{\perp}$.
A more precise result is the following (see e.g.\ \cite[Theorem~1.44]{IarrobinoKanev}, see also \cite{Blekherman}).
\begin{proposition}
  Assume that $f\in R_d$ is not a power of a linear form.
  Then its apolar ideal $f^\perp$ is 
  generated by two forms $g,g'$ such that $d = \deg g+\deg g'-2$ and
  $\mathrm{gcd}(g,g') = 1$.
  Conversely, any two such forms generate an ideal $f^\perp$ for 
  some projectively unique
  $f\in R_d$.  
\end{proposition}
We say that $f\in R_d$ is \emph{ generated in generic degrees} 
if  
  $(\deg g,\deg g') = (\lceil\frac {d+1}2\rceil,\lfloor\frac {d+3}2\rfloor)$.
The forms that are not generated in generic degrees form a subvariety of $R_d$,
which has codimension $1$ if $d$ is even and codimension $2$ if $d$ is odd.
Indeed this subvariety is defined by the maximal minors 
of the intermediate catalecticant matrix $A^{d,\lfloor\frac{d+1}{2}\rfloor}$ of size 
$\lceil\frac {d+1}2\rceil  \times  \lfloor\frac {d+3}2\rfloor$.

\medskip

Let $f\in R_d=\KK[x,y]_d$ be a binary form of degree $d$. 
A classical result is the following: 
\begin{lemma}[Apolarity Lemma]
\label{apolarity}
  Assume $f\in R_d$ and let $\ell_i\in R_1$ be distinct linear forms for $1\le i\le r$. 
There are coefficients $\alpha_i\in \KK$ such that $f=\sum_{i=1}^r \alpha_i(\ell_i)^d$ if and only 
if the operator  $\ell_1^\perp\circ \cdots \circ \ell_r^\perp$ is in the apolar ideal $f^\perp$.
  \end{lemma}
It follows that
a form $f$ has rank less than or equal to $r$ if and only if $(f^\perp)_r = f^\perp \cap D_r$ 
contains a form 
with all roots distinct and in $\KK$.
Recall that when $\KK=\RR$, such a form is called 
\emph{real-rooted}.

\section{Pullbacks of higher associated varieties to coincident root loci}
\label{pullbacks}

In this section we study the geometry of the pullbacks, via apolar maps, of higher associated varieties to coincident root loci, $\overline{\varPsi_{d,r}^{-1}(\mathrm{CH}_j(\Delta_{\lambda}))}$. This analysis is a key tool in the description of the algebraic boundaries that we will carry out in Section \ref{alg-boun}, and we think it is interesting in itself. 

Let $r$ be a positive integer, and 
let 
$\lambda=[\lambda_1,\ldots,\lambda_n]$ be a partition of $r$ 
of length 
  $|\lambda|=n$.
Consider
for any integer $0\leq j\leq \min\{n,r-n-1\}$ the following set of partitions
\begin{equation}\label{def-discendenti}
  \mathcal{D}_j(\lambda)=\{\lambda'=[\lambda_1-\iota_1,\lambda_2-\iota_2,\ldots,\lambda_n-\iota_n]: \iota_i\in\{0,1\},\ \iota_1+\cdots+\iota_n=j\},
  \end{equation}
and let us fix $\lambda'\in\mathcal{D}_j(\lambda)$.
Let $\Delta_\lambda\subset\PP^r$ (resp. $\Delta_{\lambda'}\subset\PP^{r-j}$)
be the coincident root locus corresponding to the partition $\lambda$ (resp. $\lambda'$).
Let
$d=r+n-j$ and 
consider the apolar maps:
\begin{align*}
\varPsi_{d,r} &:\PP^{d}\dashrightarrow \Gr(r - n - 1 + j,r) , \\
 \varPsi_{d,r-j} &:\PP^{d}\dashrightarrow \Gr(r - j - n - 1,r-j) .
\end{align*}
We consider the $j$-th higher associated variety of $\Delta_\lambda\subset\PP^{r}$ (see Subsection \ref{HAS-section}),
\[\mathrm{CH}_j(\Delta_{\lambda})\subset \Gr(r - n - 1 + j,r),\]
which is a hypersurface if and only if $j\le n-m_1(\lambda)$, where
$m_1(\lambda)$ is the number of $1$ in the partition $\lambda$.
Set $n'=|\lambda'|\le n$
and consider also the irreducible variety
associated to $\Delta_{\lambda'}\subset\PP^{r-j}$  (see Definition~\ref{defVarCH} and Example~\ref{eseCHj}),
\[\mathrm{\Xi}_0^{r-j-n-1}(\Delta_{\lambda'})\subset \Gr(r - j - n - 1,r-j). \]
By applying formula
  \eqref{codimCH0gen}
of Remark~\ref{codimAspettataCH0}
 we compute that the codimension of   $\mathrm{\Xi}_0^{r-j-n-1}(\Delta_{\lambda'})$ 
 in 
  $\Gr(r - j - n - 1,r-j)$ is
\begin{equation}\label{formula-codimensione}
\codim_{\PP^{r-j}}(\Delta_{\lambda'}) - (r-j-n-1)=(r-j)-{n'}-(r-j-n-1)=n-{n'}+1. 
\end{equation}

\begin{theorem}\label{lemmaEstivo}
Let notation be as above. Then
the following formula holds set-theoretically:
\begin{equation}\label{eqEstiva}
\overline{\varPsi_{d,r}^{-1}(\mathrm{CH}_j(\Delta_{\lambda}))}  =
\bigcup_{\lambda'\in \mathcal{D}_j(\lambda)} \overline{\varPsi_{d,r-j}^{-1}(\mathrm{\Xi}_0^{r-j-n-1}(\Delta_{\lambda'}))}  .
\end{equation}  
\end{theorem}

\begin{proof} 
We first show the inclusion $\subseteq$.
Let $f\in {\varPsi_{d,r}^{-1}(\mathrm{CH}_j(\Delta_\lambda))}$.
Then $ \varPsi_{d,r}(f) = (f^{\perp})_{r}$ 
contains a  point 
 $h_0=\ell_1^{\lambda_1}\cdots\ell^{\lambda_n}_n\in \Delta_\lambda$ 
 such that 
   $L=T_{h_0}\Delta_\lambda \cap (f^\perp)_{r}$ is a projective linear space of dimension $J\geq j$.
Thus there exist  $J+1$ linearly independent forms $q_0,\ldots,q_J$ of degree $n$ 
such that 
\begin{gather*}
 q_0\ell_1^{\lambda_1-1}\cdots\ell^{\lambda_n-1}_n \in L , \\
 \vdots \\
 q_{J}\ell_1^{\lambda_1-1}\cdots\ell^{\lambda_n-1}_n \in L .
\end{gather*}
Consider now the form $$p=\ell_1^{\lambda_1-1}\cdots\ell^{\lambda_n-1}_n(f),$$
  which has degree 
  $d-(r-n) = (r+n-j)-(r-n) = 2 n - j $.
Of course $q_0,\ldots,q_J$ annihilates $p$, yielding that the dimension of 
  $(p^\perp)_n=\PP(\ker A_p^{2n-j,n})$ is at least $J$,
  where $A_p^{2n-j,n}$
  is the catalecticant matrix of $p$ of size $(n-j+1)\times (n+1)$.
  This means that the rank of  $A_p^{2n-j,n}$ 
  is at most $n-J$. 
  
  Let us consider the catalecticant matrix $A^{2n-j,n-J}_p$ of $p$ of size $(n+(J-j)+1)\times (n-J+1)$.
  A well-known property of the catalecticant matrices
   (see e.g. \cite[Proposition~9.7]{harris-firstcourse} or \cite[Theorem~1.56]{IarrobinoKanev}) 
  implies that the ideals generated by the minors of order $n-J+1$
   of the 
   \emph{generic} catalecticant
   matrices $A^{2n-j,n}$ and 
   $A^{2n-j,n-J}$
   are the same (both define the $(n-J+1)$-secant variety of the rational normal curve in $\PP^{2n-j}$).
   Thus,
   we deduce that 
  the rank of $A^{2n-j,n-J}_p$ is at most $n-J$ as well,
  and hence that 
  there exists a form $q$ of degree $n-J$ in $(p^\perp)_{n-J}$.
  Let $\tilde{q}=q\ell_1^{\lambda_1-1}\cdots\ell^{\lambda_n-1}_n$,
  which 
  of course belongs to $(f^\perp)_{r-J}$.
 From dimensional reasons it follows that 
\[
 L = \{ \tilde{q} h: h \mbox{ binary  form of degree }J\} ,
\]
and  since $h_0=\ell_1^{\lambda_1}\cdots\ell^{\lambda_n}_n\in L$ we conclude 
that $q$ divides $\ell_1\cdots\ell_n$.
Thus we have shown that there exists $\tilde{\lambda}\in\mathcal{{D}}_J(\lambda)$
such that 
$\tilde{q}\in \Delta_{\tilde{\lambda}}$.
By multiplying $\tilde{q}$ by $J-j$ suitable forms among $\{\ell_1,\ldots,\ell_n\}$
we can find an element in $(f^{\perp})_{r-j}\cap \Delta_{\lambda'}$ with $\lambda'\in\mathcal{{D}}_j(\lambda)$.
\medskip 

 We now show the inclusion $\supseteq$. 
Let $\lambda'=(\lambda'_1,\ldots,\lambda'_{n'})\in\mathcal{D}_j(\lambda)$, ${n'}\leq n$, and 
$f\in {\varPsi_{d,r-j}^{-1}(\mathrm{\Xi}_0^{r-j-n-1}(\Delta_{\lambda'}))}$.
Then $\psi_{d,r-j}(f) = (f^{\perp})_{r-j}$
intersects $ \Delta_{\lambda'}$ in a point $q$, say 
$q=\ell_1^{\lambda'_1}\cdots \ell_{{n'}}^{\lambda'_{{n'}}}$.
Let us consider 
  the $j$-dimensional linear subspace
 \[L=\{q h: h\mbox{ binary form of degree }j\} \subseteq (f^{\perp})_{r} \subset \PP^{r} .\]
 Clearly the intersection $L\cap\Delta_{\lambda}$ is not empty and 
 let $\tilde{q} = \ell_1^{\lambda_1}\cdots \ell_n^{\lambda_n}$ denote one of its elements. 
 (Notice that if all the $\ell_i$ are distinct, then 
 the cardinality of $L\cap\Delta_{\lambda}$ is the number $m(\lambda',\lambda)$, defined in \eqref{def-multiplicity} below.)
 The tangent space 
 $T_{\tilde{q}}(\Delta_{\lambda})$
 to  $\Delta_{\lambda}$ at the point $\tilde{q}$ consists 
 of the forms which are divisible by $\ell_1^{\lambda_1-1}\cdots \ell_n^{\lambda_n-1}$.
 Hence $T_{\tilde{q}}(\Delta_{\lambda})$ contains $L$, and it follows that 
 $\psi_{d,r}(f) = (f^{\perp})_{r}\in \CH_{j}(\Delta_{\lambda})$.
 This concludes the proof.
\end{proof}

\begin {remark}\label{duali}
As  a consequence of \cite[Corollary 2.3]{LeeSturmfels}, 
we have
  \begin{equation}\label{vera?}
 \overline{\varPsi_{d,r-j}^{-1}(\Xi_0^{r-j-n-1}(\Delta_{\lambda'}))} = 
 (\Delta_{(\lambda'_1+1,\ldots,\lambda'_{n'}+1, 1,\ldots,1)})^{\vee} , 
  \end{equation}
  where the last $1$ occur $n-n'$ times, and ${n'} = |\lambda'|\le n$.
  In particular, the variety in \eqref{vera?} has codimension $1$ if and only if $n'=n$. In this case, since 
   $  \Xi_0^{r-j-n-1}(\Delta_{\lambda'}) = \mathrm{CH}_0(\Delta_{\lambda'})$ (see Example~\ref{eseCHj}), we obtain  
\begin{equation}\label{eq-duali-CH}
\overline{\varPsi_{d,r-j}^{-1}(\Xi_0^{r-j-n-1}(\Delta_{\lambda'}))} = 
\overline{\varPsi_{d,r-j}^{-1}(\mathrm{CH}_0(\Delta_{\lambda'}))} = (\Delta_{(\lambda'_1+1,\ldots,\lambda'_n+1)})^{\vee} .
  \end{equation}
\end{remark}

\begin {corollary}\label{soluzione-del-problemino}
Let notation be as above. We have
\begin{equation}
\codim_{\PP^d}(\overline{\varPsi_{d,r}^{-1}(\mathrm{CH}_j(\Delta_{\lambda}))}) = 1
\mbox{ if and only if } 
\codim_{\Gr(r - n - 1 + j,r)}(\mathrm{CH}_j(\Delta_{\lambda})) = 1 .
\end{equation}
\end{corollary}
\begin{proof}
It follows from \cite[Corollary~7.3]{Katz} 
that
   the codimension of $ (\Delta_{(\lambda'_1+1,\ldots,\lambda'_{n'}+1, 1,\ldots,1)})^{\vee}$ in $\PP^d$ is $n-{n'}+1$.
   Thus, by  \eqref{vera?} and \eqref{formula-codimensione}, we have  
   \begin{equation*} 
 \codim_{\PP^d}(\overline{\varPsi_{d,r-j}^{-1}(\Xi_0^{r-j-n-1}(\Delta_{\lambda'}))})
 = n-n'+1 
 = \codim_{\GG(r-j-n-1,r)}(\Xi_0^{r-j-n-1}(\Delta_{\lambda'})).
\end{equation*}
   Therefore, from Theorem~\ref{lemmaEstivo} we deduce that 
   $\overline{\varPsi_{d,r}^{-1}(\mathrm{CH}_j(\Delta_{\lambda}))} $ has codimension $1$ in $\PP^d$
   if and only if $\mathcal D_j(\lambda)$ contains a partition $\lambda'$ of length $n'=n$, that is 
   if and only if $j\le n-m_1(\lambda)$. But this is exactly the condition for $\CH_j(\Delta_{\lambda})$ to be a hypersurface in $\Gr(r - n - 1 + j,r)$.
\end{proof}

For any integer $0\leq j\leq n - m_1(\lambda)$, consider now
the following subset of  the set $\mathcal{D}_j(\lambda)$ defined in \eqref{def-discendenti}: 
\begin{equation}\label{def-figli}
  \mathcal{F}_j(\lambda)=\{\lambda'=[\lambda_1-\iota_1,\lambda_2-\iota_2,\ldots,\lambda_n-\iota_n]: \iota_i\in\{0,1\},\ \iota_1+\cdots+\iota_n=j, |\lambda'|=n\} .
  \end{equation}
Then we have the following (see also Table~\ref{tabella-conj-fin}):
 \begin{corollary}\label{pippo} Let notation be as above. 
 If $j\leq n - m_1(\lambda)$, then
the following formula holds set-theoretically:
\begin{equation}\label{formulona}
\overline{\varPsi_{d,r}^{-1}(\mathrm{CH}_j(\Delta_{\lambda}))} = 
\bigcup_{\lambda'\in \mathcal{F}_j(\lambda)} \overline{\varPsi_{d,r-j}^{-1}(\mathrm{CH}_0(\Delta_{\lambda'})}) 
= \bigcup_{\lambda'\in \mathcal{F}_j(\lambda)} (\Delta_{(\lambda'_1+1,\ldots,\lambda'_n+1)})^{\vee}.
\end{equation}  
\end{corollary}
\begin{proof}
Since $j\leq n - m_1(\lambda)$, then the left-hand side of 
  \eqref{eqEstiva} is a hypersurface. Hence we can exclude from the right-hand side all the components of codimension higher than 1. By Remark \ref{duali}, this corresponds to ask that $|\lambda'|=n$, and 
  in this case 
  we apply formula \eqref{eq-duali-CH}.
\end{proof}

\subsection{On the multiplicities of the components}
Here we try to give a more precise geometric description of the pullbacks, via apolar maps, of higher associated varieties to coincident root loci.  In particular we study the multiplicity of the components which appear in formula \eqref{eqEstiva} of Theorem \ref{lemmaEstivo}.

For any $\lambda'=[\lambda'_1,\ldots,\lambda'_{n'}]\in \mathcal{D}_j(\lambda)$ (see \eqref{def-discendenti}), 
we define the {\it multiplicity of $\lambda'$ with respect to $\lambda$} as follows:
\begin{equation}\label{def-multiplicity}
\begin{split}
  m(\lambda',\lambda) =  \#\{(\iota_1,\ldots,\iota_n)&: \iota_i\in\{0,1\},\ \iota_1+\cdots+\iota_n=j,
  \\ &  [\lambda'_1+\iota_1,\ldots,\lambda'_{n'}+\iota_{n'},\iota_{n'+1},\ldots,\iota_n]=[\lambda_1,\ldots,\lambda_n]\}.
 \end{split}
 \end{equation}
In particular, if $\lambda'=[\lambda'_1,\ldots,\lambda'_n]\in \mathcal{F}_j(\lambda)$ (see \eqref{def-figli}), then the multiplicity is
\[m(\lambda',\lambda)=\#\{(\iota_1,\ldots,\iota_n): \iota_i\in\{0,1\},\ \iota_1+\cdots+\iota_n=j,\ [\lambda'_1+\iota_1,\ldots,\lambda'_n+\iota_n]=[\lambda_1,\ldots,\lambda_n]\} .\] 

Notice that the definition of $m(\lambda',\lambda)$ is motivated by the last part of the proof of  Theorem \ref{lemmaEstivo}.

Based on a 
number of experimental verifications (see Remark~\ref{remVerifiche} and the examples
 below)
and some general considerations
on the singular loci of the higher associated varieties, we formulate here the following:
 \begin{conjecture}\label{conje}
 If $j\leq n - m_1(\lambda)$, then
the following formula holds scheme-theoretically:
\begin{equation*} 
\overline{\varPsi_{d,r}^{-1}(\mathrm{CH}_j(\Delta_{\lambda}))} = 
\bigcup_{\lambda'\in \mathcal{F}_j(\lambda)} {m(\lambda',\lambda)}\, \overline{\varPsi_{d,r-j}^{-1}(\mathrm{CH}_0(\Delta_{\lambda'}))}
= \bigcup_{\lambda'\in \mathcal{F}_j(\lambda)} {m(\lambda',\lambda)}\, (\Delta_{(\lambda'_1+1,\ldots,\lambda'_n+1)})^{\vee}.
\end{equation*}  
\end{conjecture}

\begin{remark}\label{polarDegrees}  
As an immediate consequence of Conjecture~\ref{conje}, together with the Oeding's formula \eqref{luke},
we deduce an explicit formula for the polar degrees $\delta_j(\Delta_{\lambda})$ of any coincident root locus $\Delta_\lambda$ (see Subsection~\ref{HAS-section}). 
In fact, if $j> n - m_1(\lambda)$ (where $n = |\lambda|$) then $\delta_j(\Delta_{\lambda}) = 0$; while  if $j\leq n - m_1(\lambda)$  and since $\psi_{d,r}$ is defined by forms of degree $d-r+1=n-j+1$,
then Conjecture~\ref{conje} gives 
\begin{align}
 \delta_j(\Delta_{\lambda}) 
&= \frac{1}{n-j+1} \sum_{\lambda'\in \mathcal{F}_j(\lambda)} {m(\lambda',\lambda)}\,  \frac{(n+1)!}{m_1(\lambda')! m_2(\lambda')! \cdots}\lambda_1' \lambda_2'\cdots \lambda_n' \\
\nonumber &=\frac{n+1}{n-j+1} \sum_{\lambda'\in \mathcal{F}_j(\lambda)} {m(\lambda',\lambda)} \deg(\Delta_{\lambda'}), 
\end{align}
where we denote by
$m_j(\lambda')$ the number of $j$ in the partition $\lambda'$. 
\end{remark}

\begin{remark}\label{remVerifiche}
We have verified {the validity} of the Conjecture~\ref{conje} for all partitions $\lambda$ of $r\leq 7$.
This has been done using the software
{\sc Macaulay2} \cite{macaulay2} with the packages \emph{CoincidentRootLoci} \cite{CRL-package} and 
\emph{Resultants} \cite{Resultants-package}.
We write out in Table~\ref{tabella-conj-fin}  the explicit formulas for these partitions. 

For the reader who wants to verify other cases, 
we also point out that
the  
polar degrees of coincident root loci can be calculated using 
the aforementioned packages; for low dimensions $r\leq7$, see also \cite[3rd col. of Table~1]{LeeSturmfels}.
\begin{table}[htbp]
\centering 
\footnotesize 
\tabcolsep=0.7pt 
\begin{adjustbox}{width=\textwidth}
\begin{tabular}{ll}
$\lambda = (3)$ & \begin{tabular}{l} $\overline{\varPsi_{4,3}^{-1}(\mathrm{CH}_0(\Delta_{\lambda}))} = {(\Delta_{(4)})^{\vee}} $,\ $\overline{\varPsi_{3,3}^{-1}(\mathrm{CH}_1(\Delta_{\lambda}))} = {(\Delta_{(3)})^{\vee}} $\end{tabular} \\ $\lambda = (2,1)$ & \begin{tabular}{l} $\overline{\varPsi_{5,3}^{-1}(\mathrm{CH}_0(\Delta_{\lambda}))} = {(\Delta_{(3,2)})^{\vee}} $,\ $\overline{\varPsi_{4,3}^{-1}(\mathrm{CH}_1(\Delta_{\lambda}))} = 2\cdot{(\Delta_{(2,2)})^{\vee}} $\end{tabular} \\ $\lambda = (4)$ & \begin{tabular}{l} $\overline{\varPsi_{5,4}^{-1}(\mathrm{CH}_0(\Delta_{\lambda}))} = {(\Delta_{(5)})^{\vee}} $,\ $\overline{\varPsi_{4,4}^{-1}(\mathrm{CH}_1(\Delta_{\lambda}))} = {(\Delta_{(4)})^{\vee}} $\end{tabular} \\ $\lambda = (3,1)$ & \begin{tabular}{l} $\overline{\varPsi_{6,4}^{-1}(\mathrm{CH}_0(\Delta_{\lambda}))} = {(\Delta_{(4,2)})^{\vee}} $,\ $\overline{\varPsi_{5,4}^{-1}(\mathrm{CH}_1(\Delta_{\lambda}))} = {(\Delta_{(3,2)})^{\vee}} $\end{tabular} \\ $\lambda = (2,2)$ & \begin{tabular}{l} $\overline{\varPsi_{6,4}^{-1}(\mathrm{CH}_0(\Delta_{\lambda}))} = {(\Delta_{(3,3)})^{\vee}} $,\ $\overline{\varPsi_{5,4}^{-1}(\mathrm{CH}_1(\Delta_{\lambda}))} = {(\Delta_{(3,2)})^{\vee}} $,\ $\overline{\varPsi_{4,4}^{-1}(\mathrm{CH}_2(\Delta_{\lambda}))} = {(\Delta_{(2,2)})^{\vee}} $\end{tabular} \\ $\lambda = (2,1,1)$ & \begin{tabular}{l} $\overline{\varPsi_{7,4}^{-1}(\mathrm{CH}_0(\Delta_{\lambda}))} = {(\Delta_{(3,2,2)})^{\vee}} $,\ $\overline{\varPsi_{6,4}^{-1}(\mathrm{CH}_1(\Delta_{\lambda}))} = 3\cdot{(\Delta_{(2,2,2)})^{\vee}} $\end{tabular} \\ $\lambda = (5)$ & \begin{tabular}{l} $\overline{\varPsi_{6,5}^{-1}(\mathrm{CH}_0(\Delta_{\lambda}))} = {(\Delta_{(6)})^{\vee}} $,\ $\overline{\varPsi_{5,5}^{-1}(\mathrm{CH}_1(\Delta_{\lambda}))} = {(\Delta_{(5)})^{\vee}} $\end{tabular} \\ $\lambda = (4,1)$ & \begin{tabular}{l} $\overline{\varPsi_{7,5}^{-1}(\mathrm{CH}_0(\Delta_{\lambda}))} = {(\Delta_{(5,2)})^{\vee}} $,\ $\overline{\varPsi_{6,5}^{-1}(\mathrm{CH}_1(\Delta_{\lambda}))} = {(\Delta_{(4,2)})^{\vee}} $\end{tabular} \\ $\lambda = (3,2)$ & \begin{tabular}{l} $\overline{\varPsi_{7,5}^{-1}(\mathrm{CH}_0(\Delta_{\lambda}))} = {(\Delta_{(4,3)})^{\vee}} $,\ $\overline{\varPsi_{6,5}^{-1}(\mathrm{CH}_1(\Delta_{\lambda}))} = {(\Delta_{(4,2)})^{\vee}} \bigcup 2\cdot{(\Delta_{(3,3)})^{\vee}} $,\ $\overline{\varPsi_{5,5}^{-1}(\mathrm{CH}_2(\Delta_{\lambda}))} = {(\Delta_{(3,2)})^{\vee}} $\end{tabular} \\ $\lambda = (3,1,1)$ & \begin{tabular}{l} $\overline{\varPsi_{8,5}^{-1}(\mathrm{CH}_0(\Delta_{\lambda}))} = {(\Delta_{(4,2,2)})^{\vee}} $,\ $\overline{\varPsi_{7,5}^{-1}(\mathrm{CH}_1(\Delta_{\lambda}))} = {(\Delta_{(3,2,2)})^{\vee}} $\end{tabular} \\ $\lambda = (2,2,1)$ & \begin{tabular}{l} $\overline{\varPsi_{8,5}^{-1}(\mathrm{CH}_0(\Delta_{\lambda}))} = {(\Delta_{(3,3,2)})^{\vee}} $,\ $\overline{\varPsi_{7,5}^{-1}(\mathrm{CH}_1(\Delta_{\lambda}))} = 2\cdot{(\Delta_{(3,2,2)})^{\vee}} $,\ $\overline{\varPsi_{6,5}^{-1}(\mathrm{CH}_2(\Delta_{\lambda}))} = 3\cdot{(\Delta_{(2,2,2)})^{\vee}} $\end{tabular} \\ $\lambda = (2,1,1,1)$ & \begin{tabular}{l} $\overline{\varPsi_{9,5}^{-1}(\mathrm{CH}_0(\Delta_{\lambda}))} = {(\Delta_{(3,2,2,2)})^{\vee}} $,\ $\overline{\varPsi_{8,5}^{-1}(\mathrm{CH}_1(\Delta_{\lambda}))} = 4\cdot{(\Delta_{(2,2,2,2)})^{\vee}} $\end{tabular} \\ $\lambda = (6)$ & \begin{tabular}{l} $\overline{\varPsi_{7,6}^{-1}(\mathrm{CH}_0(\Delta_{\lambda}))} = {(\Delta_{(7)})^{\vee}} $,\ $\overline{\varPsi_{6,6}^{-1}(\mathrm{CH}_1(\Delta_{\lambda}))} = {(\Delta_{(6)})^{\vee}} $\end{tabular} \\ $\lambda = (5,1)$ & \begin{tabular}{l} $\overline{\varPsi_{8,6}^{-1}(\mathrm{CH}_0(\Delta_{\lambda}))} = {(\Delta_{(6,2)})^{\vee}} $,\ $\overline{\varPsi_{7,6}^{-1}(\mathrm{CH}_1(\Delta_{\lambda}))} = {(\Delta_{(5,2)})^{\vee}} $\end{tabular} \\ $\lambda = (4,2)$ & \begin{tabular}{l} $\overline{\varPsi_{8,6}^{-1}(\mathrm{CH}_0(\Delta_{\lambda}))} = {(\Delta_{(5,3)})^{\vee}} $,\ $\overline{\varPsi_{7,6}^{-1}(\mathrm{CH}_1(\Delta_{\lambda}))} = {(\Delta_{(5,2)})^{\vee}} \bigcup {(\Delta_{(4,3)})^{\vee}} $,\ $\overline{\varPsi_{6,6}^{-1}(\mathrm{CH}_2(\Delta_{\lambda}))} = {(\Delta_{(4,2)})^{\vee}} $\end{tabular} \\ $\lambda = (4,1,1)$ & \begin{tabular}{l} $\overline{\varPsi_{9,6}^{-1}(\mathrm{CH}_0(\Delta_{\lambda}))} = {(\Delta_{(5,2,2)})^{\vee}} $,\ $\overline{\varPsi_{8,6}^{-1}(\mathrm{CH}_1(\Delta_{\lambda}))} = {(\Delta_{(4,2,2)})^{\vee}} $\end{tabular} \\ $\lambda = (3,3)$ & \begin{tabular}{l} $\overline{\varPsi_{8,6}^{-1}(\mathrm{CH}_0(\Delta_{\lambda}))} = {(\Delta_{(4,4)})^{\vee}} $,\ $\overline{\varPsi_{7,6}^{-1}(\mathrm{CH}_1(\Delta_{\lambda}))} = {(\Delta_{(4,3)})^{\vee}} $,\ $\overline{\varPsi_{6,6}^{-1}(\mathrm{CH}_2(\Delta_{\lambda}))} = {(\Delta_{(3,3)})^{\vee}} $\end{tabular} \\ $\lambda = (3,2,1)$ & \begin{tabular}{l} $\overline{\varPsi_{9,6}^{-1}(\mathrm{CH}_0(\Delta_{\lambda}))} = {(\Delta_{(4,3,2)})^{\vee}} $,\ $\overline{\varPsi_{8,6}^{-1}(\mathrm{CH}_1(\Delta_{\lambda}))} = 2\cdot{(\Delta_{(4,2,2)})^{\vee}} \bigcup 2\cdot{(\Delta_{(3,3,2)})^{\vee}} $,\ $\overline{\varPsi_{7,6}^{-1}(\mathrm{CH}_2(\Delta_{\lambda}))} = 2\cdot{(\Delta_{(3,2,2)})^{\vee}} $\end{tabular} \\ $\lambda = (3,1,1,1)$ & \begin{tabular}{l} $\overline{\varPsi_{10,6}^{-1}(\mathrm{CH}_0(\Delta_{\lambda}))} = {(\Delta_{(4,2,2,2)})^{\vee}} $,\ $\overline{\varPsi_{9,6}^{-1}(\mathrm{CH}_1(\Delta_{\lambda}))} = {(\Delta_{(3,2,2,2)})^{\vee}} $\end{tabular} \\ 
$\lambda = (2,2,2)$ & \begin{tabular}{l} $\overline{\varPsi_{9,6}^{-1}(\mathrm{CH}_0(\Delta_{\lambda}))} = {(\Delta_{(3,3,3)})^{\vee}} $,\ $\overline{\varPsi_{8,6}^{-1}(\mathrm{CH}_1(\Delta_{\lambda}))} = {(\Delta_{(3,3,2)})^{\vee}} $,\ $\overline{\varPsi_{7,6}^{-1}(\mathrm{CH}_2(\Delta_{\lambda}))} = {(\Delta_{(3,2,2)})^{\vee}} $,\ $\overline{\varPsi_{6,6}^{-1}(\mathrm{CH}_3(\Delta_{\lambda}))} = {(\Delta_{(2,2,2)})^{\vee}} $\end{tabular} \\ $\lambda = (2,2,1,1)$ & \begin{tabular}{l} $\overline{\varPsi_{10,6}^{-1}(\mathrm{CH}_0(\Delta_{\lambda}))} = {(\Delta_{(3,3,2,2)})^{\vee}} $,\ $\overline{\varPsi_{9,6}^{-1}(\mathrm{CH}_1(\Delta_{\lambda}))} = 3\cdot{(\Delta_{(3,2,2,2)})^{\vee}} $,\ $\overline{\varPsi_{8,6}^{-1}(\mathrm{CH}_2(\Delta_{\lambda}))} = 6\cdot{(\Delta_{(2,2,2,2)})^{\vee}} $\end{tabular} \\ $\lambda = (2,1,1,1,1)$ & \begin{tabular}{l} $\overline{\varPsi_{11,6}^{-1}(\mathrm{CH}_0(\Delta_{\lambda}))} = {(\Delta_{(3,2,2,2,2)})^{\vee}} $,\ $\overline{\varPsi_{10,6}^{-1}(\mathrm{CH}_1(\Delta_{\lambda}))} = 5\cdot{(\Delta_{(2,2,2,2,2)})^{\vee}} $\end{tabular} \\ $\lambda = (7)$ & \begin{tabular}{l} $\overline{\varPsi_{8,7}^{-1}(\mathrm{CH}_0(\Delta_{\lambda}))} = {(\Delta_{(8)})^{\vee}} $,\ $\overline{\varPsi_{7,7}^{-1}(\mathrm{CH}_1(\Delta_{\lambda}))} = {(\Delta_{(7)})^{\vee}} $\end{tabular} \\ $\lambda = (6,1)$ & \begin{tabular}{l} $\overline{\varPsi_{9,7}^{-1}(\mathrm{CH}_0(\Delta_{\lambda}))} = {(\Delta_{(7,2)})^{\vee}} $,\ $\overline{\varPsi_{8,7}^{-1}(\mathrm{CH}_1(\Delta_{\lambda}))} = {(\Delta_{(6,2)})^{\vee}} $\end{tabular} \\ $\lambda = (5,2)$ & \begin{tabular}{l} $\overline{\varPsi_{9,7}^{-1}(\mathrm{CH}_0(\Delta_{\lambda}))} = {(\Delta_{(6,3)})^{\vee}} $,\ $\overline{\varPsi_{8,7}^{-1}(\mathrm{CH}_1(\Delta_{\lambda}))} = {(\Delta_{(6,2)})^{\vee}} \bigcup {(\Delta_{(5,3)})^{\vee}} $,\ $\overline{\varPsi_{7,7}^{-1}(\mathrm{CH}_2(\Delta_{\lambda}))} = {(\Delta_{(5,2)})^{\vee}} $\end{tabular} \\ $\lambda = (5,1,1)$ & \begin{tabular}{l} $\overline{\varPsi_{10,7}^{-1}(\mathrm{CH}_0(\Delta_{\lambda}))} = {(\Delta_{(6,2,2)})^{\vee}} $,\ $\overline{\varPsi_{9,7}^{-1}(\mathrm{CH}_1(\Delta_{\lambda}))} = {(\Delta_{(5,2,2)})^{\vee}} $\end{tabular} \\ $\lambda = (4,3)$ & \begin{tabular}{l} $\overline{\varPsi_{9,7}^{-1}(\mathrm{CH}_0(\Delta_{\lambda}))} = {(\Delta_{(5,4)})^{\vee}} $,\ $\overline{\varPsi_{8,7}^{-1}(\mathrm{CH}_1(\Delta_{\lambda}))} = {(\Delta_{(5,3)})^{\vee}} \bigcup 2\cdot{(\Delta_{(4,4)})^{\vee}} $,\ $\overline{\varPsi_{7,7}^{-1}(\mathrm{CH}_2(\Delta_{\lambda}))} = {(\Delta_{(4,3)})^{\vee}} $\end{tabular} \\ $\lambda = (4,2,1)$ & \begin{tabular}{l} $\overline{\varPsi_{10,7}^{-1}(\mathrm{CH}_0(\Delta_{\lambda}))} = {(\Delta_{(5,3,2)})^{\vee}} $,\ $\overline{\varPsi_{9,7}^{-1}(\mathrm{CH}_1(\Delta_{\lambda}))} = 2\cdot{(\Delta_{(5,2,2)})^{\vee}} \bigcup {(\Delta_{(4,3,2)})^{\vee}} $,\ $\overline{\varPsi_{8,7}^{-1}(\mathrm{CH}_2(\Delta_{\lambda}))} = 2\cdot{(\Delta_{(4,2,2)})^{\vee}} $\end{tabular} \\ $\lambda = (4,1,1,1)$ & \begin{tabular}{l} $\overline{\varPsi_{11,7}^{-1}(\mathrm{CH}_0(\Delta_{\lambda}))} = {(\Delta_{(5,2,2,2)})^{\vee}} $,\ $\overline{\varPsi_{10,7}^{-1}(\mathrm{CH}_1(\Delta_{\lambda}))} = {(\Delta_{(4,2,2,2)})^{\vee}} $\end{tabular} \\ $\lambda = (3,3,1)$ & \begin{tabular}{l} $\overline{\varPsi_{10,7}^{-1}(\mathrm{CH}_0(\Delta_{\lambda}))} = {(\Delta_{(4,4,2)})^{\vee}} $,\ $\overline{\varPsi_{9,7}^{-1}(\mathrm{CH}_1(\Delta_{\lambda}))} = {(\Delta_{(4,3,2)})^{\vee}} $,\ $\overline{\varPsi_{8,7}^{-1}(\mathrm{CH}_2(\Delta_{\lambda}))} = {(\Delta_{(3,3,2)})^{\vee}} $\end{tabular} \\ $\lambda = (3,2,2)$ & \begin{tabular}{l} $\overline{\varPsi_{10,7}^{-1}(\mathrm{CH}_0(\Delta_{\lambda}))} = {(\Delta_{(4,3,3)})^{\vee}} $,\ $\overline{\varPsi_{9,7}^{-1}(\mathrm{CH}_1(\Delta_{\lambda}))} = {(\Delta_{(4,3,2)})^{\vee}} \bigcup 3\cdot{(\Delta_{(3,3,3)})^{\vee}} $,\ $\overline{\varPsi_{8,7}^{-1}(\mathrm{CH}_2(\Delta_{\lambda}))} = 2\cdot{(\Delta_{(3,3,2)})^{\vee}} \bigcup {(\Delta_{(4,2,2)})^{\vee}} $,\ $\overline{\varPsi_{7,7}^{-1}(\mathrm{CH}_3(\Delta_{\lambda}))} = {(\Delta_{(3,2,2)})^{\vee}} $\end{tabular} \\ $\lambda = (3,2,1,1)$ & \begin{tabular}{l} $\overline{\varPsi_{11,7}^{-1}(\mathrm{CH}_0(\Delta_{\lambda}))} = {(\Delta_{(4,3,2,2)})^{\vee}} $,\ $\overline{\varPsi_{10,7}^{-1}(\mathrm{CH}_1(\Delta_{\lambda}))} = 3\cdot{(\Delta_{(4,2,2,2)})^{\vee}} \bigcup 2\cdot{(\Delta_{(3,3,2,2)})^{\vee}} $,\ $\overline{\varPsi_{9,7}^{-1}(\mathrm{CH}_2(\Delta_{\lambda}))} = 3\cdot{(\Delta_{(3,2,2,2)})^{\vee}} $\end{tabular} \\ $\lambda = (3,1,1,1,1)$ & \begin{tabular}{l} $\overline{\varPsi_{12,7}^{-1}(\mathrm{CH}_0(\Delta_{\lambda}))} = {(\Delta_{(4,2,2,2,2)})^{\vee}} $,\ $\overline{\varPsi_{11,7}^{-1}(\mathrm{CH}_1(\Delta_{\lambda}))} = {(\Delta_{(3,2,2,2,2)})^{\vee}} $\end{tabular} \\ $\lambda = (2,2,2,1)$ & \begin{tabular}{l} $\overline{\varPsi_{11,7}^{-1}(\mathrm{CH}_0(\Delta_{\lambda}))} = {(\Delta_{(3,3,3,2)})^{\vee}} $,\ $\overline{\varPsi_{10,7}^{-1}(\mathrm{CH}_1(\Delta_{\lambda}))} = 2\cdot{(\Delta_{(3,3,2,2)})^{\vee}} $,\ $\overline{\varPsi_{9,7}^{-1}(\mathrm{CH}_2(\Delta_{\lambda}))} = 3\cdot{(\Delta_{(3,2,2,2)})^{\vee}} $,\ $\overline{\varPsi_{8,7}^{-1}(\mathrm{CH}_3(\Delta_{\lambda}))} = 4\cdot{(\Delta_{(2,2,2,2)})^{\vee}} $\end{tabular} \\ $\lambda = (2,2,1,1,1)$ & \begin{tabular}{l} $\overline{\varPsi_{12,7}^{-1}(\mathrm{CH}_0(\Delta_{\lambda}))} = {(\Delta_{(3,3,2,2,2)})^{\vee}} $,\ 
$\overline{\varPsi_{11,7}^{-1}(\mathrm{CH}_1(\Delta_{\lambda}))} = 4\cdot{(\Delta_{(3,2,2,2,2)})^{\vee}} $,\ $\overline{\varPsi_{10,7}^{-1}(\mathrm{CH}_2(\Delta_{\lambda}))} = 10\cdot{(\Delta_{(2,2,2,2,2)})^{\vee}} $\end{tabular} \\ $\lambda = (2,1,1,1,1,1)$ & \begin{tabular}{l} $\overline{\varPsi_{13,7}^{-1}(\mathrm{CH}_0(\Delta_{\lambda}))} = {(\Delta_{(3,2,2,2,2,2)})^{\vee}} $,\ $\overline{\varPsi_{12,7}^{-1}(\mathrm{CH}_1(\Delta_{\lambda}))} = 6\cdot{(\Delta_{(2,2,2,2,2,2)})^{\vee}} $\end{tabular} 
\end{tabular}
\end{adjustbox}
\caption{\footnotesize Explicit formulas of Conjecture~\ref{conje} (and Corollary~\ref{pippo}) for all partitions $\lambda$ of $r$ with $3\leq r\leq 7$ and with $j\leq |\lambda| - m_1(\lambda)$.}
\label{tabella-conj-fin}
\end{table}
\end{remark}

We illustrate the formula of Conjecture~\ref{conje} with some examples.
\begin{example}
Let $\lambda=(2,1^{n-1})$. Then $\mathcal{F}(\lambda)=\{(1^{n})\}$ and $m((1^n),\lambda) = n$. 
In this special case, Conjecture~\ref{conje} gives a formula which was proved in \cite[p.~514]{LeeSturmfels}, that is
 $$
   \overline{\varPsi_{2n,n+1}^{-1}(\mathrm{CH}_1(\Delta_{(2,1^{n-1})}))} = {n}(\Delta_{(2^{n})})^{\vee}.
 $$
\end{example}

\begin{example} 
Take the partition
 $\lambda=(4,3,2,2)$ and $j=1$, hence $n=4$, $r=11$, and $d=14$.
We have $\mathcal{F}_j(\lambda)=\{(4,3,2,1),(4,2,2,2),(3,3,2,2)\}$, with
  $m((4,3,2,1),\lambda)=1$, $m((4,2,2,2),\lambda)=3$, and $m((3,3,2,2),\lambda)=2$.
The apolar map $\varPsi_{14,11} :\PP^{14}\dashrightarrow \Gr(7,11)\subset\PP^{494}$
is defined by forms of degree $4$, and 
the degree of the hypersurface $\mathrm{CH}_1(\Delta_{(4,3,2,2}))$ 
is $1740$ (calculated with {\sc Macaulay2}).
From Corollary~\ref{pippo}
it follows that $\overline{\varPsi_{14,11}^{-1}(\mathrm{CH}_1(\Delta_{(4,3,2,2)}))}$
has the following 
three components:
\[
 (\Delta_{(5,4,3,2)})^{\vee} ,\quad 
 (\Delta_{(5,3,3,3)})^{\vee} , \quad 
 (\Delta_{(4,4,3,3)})^{\vee}
\]
which have degrees, respectively, $2880$, $640$, and $1080$. 
Conjecture~\ref{conje} predicts that the first component 
occurs with multiplicity $1$, the second one with multiplicity $3$, and the 
last one with multiplicity $2$. 
This is indeed the case, since we have  
$1\cdot 2880 + 3\cdot 640 + 2\cdot 1080 = 4\cdot 1740$.
\end{example}

\section{Algebraic boundaries among typical ranks for binary forms}\label{alg-boun}
Now, thanks to the results of the previous section, we are able to describe all the algebraic boundaries for binary forms of arbitrary degree $d$.
We will consider the case of odd degree 
in Theorem~\ref{teoremaOdd}, and the case 
of even degree in Theorem~\ref{casopari}.

First, we state a preliminary result of differential topology, whose proof relies on
the property of stability of transverse intersection, see \cite[Chapter 2, Section 5-6]{Guillemin-Pollack}.
We sketch the proof for the reader's convenience.
\begin{lemma}\label{GuilleminPollak}
If $X\subset\PP^N_{\mathbb{R}}$ is an irreducible variety of codimension $c$, and 
  $\varepsilon\in[0,1)\mapsto \Pi_{\varepsilon}\in\mathbb{G}(n,\PP^N_{\mathbb{R}})$
  is 
  a one-parameter smooth family 
  of $n$-dimensional linear subspaces such that $\Pi_0$ 
  meets $X$ at a point $p$ and
  $\Pi_{\varepsilon}\cap X=\emptyset$ for each $\varepsilon>0$,  then 
  we have $\dim(T_p(X)\cap \Pi_{0}) \geq n - c + 1 $.
  \end{lemma}
  \begin{proof}[Sketch of the proof]
  We can assume that $p$ is a smooth point of $X$ 
  and that the dimension of the intersection $\Pi_0\cap X$ at $p$ is $n-c\geq 0$,
   since otherwise the claim is trivial. Under these assumptions,
   we have to show that $X$ and $\Pi_0$ meet non-transversally at $p$,
   that is 
   $\dim\langle T_p(X) , \Pi_0\rangle < N$.
   If, by contradiction, they meet transversally at $p$,
   since the property to intersect transversally 
   is stable under small deformations, we would have that also 
   $X$ and $\Pi_{\varepsilon}$  meet transversally and $X\cap \Pi_{\varepsilon}\neq \emptyset$, for any small $\varepsilon>0$.
  \end{proof}

\begin{theorem}\label{teoremaOdd}  
  If $d = 2 k -1$ is odd, then
  the algebraic boundaries for degree $d$ binary real forms 
  satisfy the following:
\begin{enumerate}
\item\label{partLeeSt} If $i=0$ then \[\partial_{\textrm{alg}}(\mathcal{R}_{d,k+i}) = (\Delta_{(3,2^{k-2})})^\vee . \] 
\item If $1\leq i\leq k-2$ then
  \[\partial_{\textrm{alg}}(\mathcal{R}_{d,k+i}) \subseteq \bigcup_{\rho} (\Delta_{\rho})^{\vee} , \] 
where 
$\rho$ runs among the partitions of $d$ 
with all parts greater than or equal to $2$
and of length between  $k-1-i$ and  $k-1$;
in particular, 
  \[\partial_{\textrm{alg}}(\mathcal{R}_{d,k+1}) \subseteq (\Delta_{(3,2^{k-2})})^\vee\cup
  (\Delta_{(4,3,2^{k-4})})^\vee\cup (\Delta_{(5,2^{k-3})})^\vee\cup (\Delta_{(3,3,3,2^{k-5})})^\vee .\]
\item\label{causaRe} If $i = k-1$, then \[\partial_{\textrm{alg}}(\mathcal{R}_{d,k+i}) = (\Delta_{(d)})^\vee . \]
  \end{enumerate}
\end{theorem}

\begin{proof}
Part \eqref{causaRe} has been shown in \cite{CausaRe},
and part \eqref{partLeeSt} is one of the results contained in \cite{LeeSturmfels}.
Now we apply the idea, given in \cite{LeeSturmfels} and developed in \cite{BS18},
to study 
the boundary $\partial_{\textrm{alg}}(\mathcal{R}_{d,k+i})$ between 
$\mathcal{R}_{d,k+i}$ and $\mathcal{R}_{d,k+i+1}\cup\cdots\cup \mathcal{R}_{d,d}$, where $1\leq i\leq k-2$.

Let $\{f_{\varepsilon}\}_{\varepsilon}$ be a continuous (smooth) family of forms going 
from $\mathcal{R}_{d,k+i}$ to $\mathcal{R}_{d,k+i+1}\cup\cdots\cup\mathcal{R}_{d,d}$ and crossing 
some irreducible component of
the boundary 
$\partial_{\textrm{alg}}(\mathcal{R}_{d,k+i})$ 
at a general
point $f_0=f$.
Thus, we assume that $f_{-\varepsilon}\in\mathcal{R}_{d,k+i}$ and 
$f_\varepsilon\in\mathcal{R}_{d,k+i+1}\cup\cdots\cup \mathcal{R}_{d,d}$ for any small $\varepsilon$ with $\varepsilon>0$.

Let $\Gr(2 i,k+i)$ be the Grassmannian of $2i$-planes in $\PP(D_{k+i})$, and
consider the apolar map
\begin{equation*}
\varPsi_{d,k+i}: \PP^d\dashrightarrow Z_{d,k+i}\subset\Gr(2 i,k+i) \subset \PP^{\binom{k+i+1}{2 i +1}-1},
\end{equation*}
which, since $i\geq 1$, is a birational map  onto its image  $Z_{d,k+i}\subset\PP^{\binom{k+i+1}{2 i +1}-1}$.
The exceptional locus of $\varPsi_{d,k+i}$ is contained in 
the locus of those forms whose annihilator is not generated in generic degrees, 
which has codimension $2$ in $\PP^d$. Therefore we can assume that $\varPsi_{d,k+i}$ 
is a local isomorphism at  $f_{\varepsilon}$ for each $\varepsilon$,
and hence it 
sends bijectively
the family $\{f_\varepsilon\}_{\varepsilon}$ into 
a continuous family $\{\Pi_{\varepsilon}\}_{\varepsilon}$ of apolar $2i$-planes.

From the Apolarity Lemma (Lemma \ref{apolarity}), we obtain that the 
$2i$-plane
$\Pi_{\varepsilon}$, with $\varepsilon<0$, 
contains a real-rooted form 
$h_\varepsilon$, while 
$\Pi_{\varepsilon}$, with $\varepsilon>0$, does not contain any real-rooted form.
The set of real-rooted forms
is a full-dimensional connected  semi-algebraic subset of $\PP^{k+i}$, 
and the Zariski closure of its topological 
boundary is the discriminant hypersurface $\Delta=\Delta_{(2,1^{k+i-2})}$.
Recall also that the singular locus of a coincident root locus 
$\Delta_\lambda$ is given by a union of $\Delta_\mu$, for suitable coarsenings $\mu$ of $\lambda$.
Thus we deduce that the limit 
$h_0=\lim_{\varepsilon\rightarrow 0^{-}} h_{\varepsilon}$ must belong to 
$\Delta $. More precisely, by applying Lemma~\ref{GuilleminPollak},
we obtain that there must exist $a\in\{0,\ldots,2 i\}$ such that 
  $h_0$ is a (smooth) point 
  of a coincident root locus $\Delta'\subseteq\Delta\subset\PP^{k+i}$
  corresponding to a partition $\lambda$ of $k+i$ of length $k+i-a-1$
  and the tangent space 
  $T_{h_0}(\Delta')$
  intersects $\Pi_0\simeq\PP^{2i}$ in a subspace $H$ of dimension at least $2i-a$ passing through~$h_0$.
  This implies that $\Pi_0\in \CH_{2i - a}(\Delta')$, 
  so that 
   $f\in\overline{\varPsi_{d,k+i}^{-1}(\CH_{2i - a}(\Delta'))}$.
  
  Thus we have that each irreducible component
  of the boundary that separates 
  $\mathcal{R}_{d,k+i}$ from $\mathcal{R}_{d,k+i+1}\cup\cdots\cup \mathcal{R}_{d,d}$ is contained in the following union
  \[
   \bigcup_{a=0}^{2 i} \bigcup_{\lambda} \overline{\varPsi_{d,k+i}^{-1}(\CH_{2i - a}(\Delta_{\lambda}))},
  \]
where $\lambda$ runs among all the partitions of $k+i$ 
of length $k+i-a-1$,
 and where moreover 
$\overline{\varPsi_{d,k+i}^{-1}(\CH_{2i - a}(\Delta_{\lambda}))}$ 
is required to have
codimension $1$.
By Corollary~\ref{soluzione-del-problemino}, this last condition is equivalent to the fact that 
$\CH_{2i - a}(\Delta_{\lambda})$ is a hypersurface in $\Gr(2 i,k+i)$,
that is such that 
$2i-a\leq k+i-a-1 - m_1(\lambda)$.
Thus, by applying Corollary~\ref{pippo}, 
we deduce that
  \begin{equation*}
    \bigcup_{a=0}^{2 i} \bigcup_{\lambda}\overline{\varPsi_{d,k+i}^{-1}(\CH_{2i - a}(\Delta_{\lambda}))}
  \subseteq  
    \bigcup_{a=0}^{2 i} \bigcup_{\mu} 
       \overline{\varPsi_{d,k-i+a }^{-1}(\CH_{0}(\Delta_{\mu}))} 
       = \bigcup_{a=0}^{2 i} \bigcup_{\nu} (\Delta_{\nu})^{\vee} = \bigcup_{\rho} (\Delta_{\rho})^{\vee},
   \end{equation*}
where $\mu$ runs among the partitions of $k-i+a$ of length $k+i-a-1$,  
$\nu$ runs among the partitions of $d$ of length $k+i-a-1$ with
parts  $\geq 2$, and 
$\rho$ runs among the partitions of $d$ of length between  $k-1-i$ and  $k-1+i$
with parts $\geq2$.
Of course we must have $a\geq i$ and the length of  
a partition $\rho$ is at most $k-1$.

Let
$$
{\mathcal{S}}_i^d = \bigcup_{\lambda} (\Delta_{\lambda})^{\vee},
$$
 where $\lambda$ runs among the 
partitions of $d$ of length $k-1-i$ and with all parts $\geq 2$.
We have shown that
$$
 \partial(\mathcal{R}_{d,k+i})\setminus \bigcup_{j=1}^{i} \partial(\mathcal{R}_{d,k+i-j}) \subseteq {\mathcal S}_i^d\cup  {\mathcal S}_{i-1}^d \cup \cdots \cup {\mathcal S}_0^d.
$$
Now, by an easy induction on $i$, we obtain
$$ \partial(\mathcal{R}_{d,k+i}) \subseteq {\mathcal S}_i^d\cup  {\mathcal S}_{i-1}^d \cup \cdots \cup {\mathcal S}_0^d ,
$$
and we conclude the proof by taking the Zariski closure.
\end{proof}

  \begin{theorem}\label{casopari}
 If $d = 2 k $ is even,  then
  the algebraic boundaries for degree $d$ binary real forms 
  satisfy the following:
\begin{enumerate}
\item\label{partLeeStEven} If $i=1$ then \[\partial_{\textrm{alg}}(\mathcal{R}_{d,k+i}) 
= (\Delta_{(3^2,2^{k-3})})^\vee \cup (\Delta_{(4,2^{k-2})})^\vee.\]
\item\label{parte2pari} If $2\leq i\leq k-1$ then  
  \[
  \partial_{\textrm{alg}}(\mathcal{R}_{d,k+i}) \subseteq  
\bigcup_{\rho} (\Delta_{\rho})^{\vee} , \]
where 
$\rho$ runs among the partitions of $d$ 
with all parts greater than or equal to $2$
and of length between  $k-i$ and  $k$;
in particular, 
  \begin{align*}
   {\partial_{\textrm{alg}}(\mathcal{R}_{d,k+2})} &\subseteq (\Delta_{(2^{k})})^\vee \cup 
   (\Delta_{(3^2,2^{k-3})})^\vee \cup (\Delta_{(4,2^{k-2})})^\vee 
  \cup (\Delta_{(6,2^{k-3})})^\vee
  \cup (\Delta_{(5,3,2^{k-4})})^\vee \\ & \quad 
  \cup (\Delta_{(4^2,2^{k-4})})^\vee \cup (\Delta_{(4,3^2,2^{k-5})})^\vee
\cup (\Delta_{(3^4,2^{k-6})})^\vee .
\end{align*}
\item\label{causaReEven} If $i = k$, then \[\partial_{\textrm{alg}}(\mathcal{R}_{d,k+i}) = (\Delta_{(d)})^\vee . \]
  \end{enumerate}
\end{theorem}
  \begin{proof}
   Thanks to \cite{LeeSturmfels} and \cite{CausaRe} we have only to show part \eqref{parte2pari}.
   The proof for this part is quite similar to that of Theorem~\ref{teoremaOdd}
   and we now sketch it.
   Consider
    a continuous  (smooth) family of degree $d$  forms $\{f_{\varepsilon}\}_{\varepsilon}$ such that 
    $f_{-\varepsilon}\in\mathcal{R}_{d,k+i}$ and 
$f_\varepsilon\in\mathcal{R}_{d,k+i+1}\cup\cdots\cup \mathcal{R}_{d,d}$ for any small $\varepsilon$ with $\varepsilon>0$,
and where we can also require that $f_{\varepsilon}\notin (\Delta_{2^k})^{\vee}$.
The birational map 
\begin{equation*}
\varPsi_{d,k+i}: \PP^d\dashrightarrow Z_{d,k+i}\subset\Gr(2 i-1,k+i) \subset \PP^{\binom{k+i+1}{2 i}-1} 
\end{equation*}
has exceptional locus contained in the hypersurface 
$(\Delta_{2^k})^{\vee}$, so it
 sends bijectively the family $\{f_\varepsilon\}_{\varepsilon}$ into 
a continuous family of apolar $(2i-1)$-planes $\{\Pi_{\varepsilon}\}_{\varepsilon}$.  From Lemma~\ref{apolarity} and Lemma~\ref{GuilleminPollak}, we deduce that $f=f_0$ must belong to the following union 
   \[
   \bigcup_{a=0}^{2 i-1} \bigcup_{\lambda} \overline{\varPsi_{d,k+i}^{-1}(\CH_{2i - a -1}(\Delta_{\lambda}))},
  \]
where $\lambda$ runs among all the partitions of $k+i$ 
of length $k+i-a-1$
 and 
{such that 
$\overline{\varPsi_{d,k+i}^{-1}(\CH_{2i - a -1}(\Delta_{\lambda}))}$ has codimension $1$.
Now we can conclude by applying 
Corollaries~\ref{soluzione-del-problemino} and \ref{pippo},} as in the proof of Theorem~\ref{teoremaOdd}.
  \end{proof}
 
  \begin{remark}\label{quasi-finito-1}
   If $d=2k\leq 8$, then the hypersurface $(\Delta_{(2^k)})^{\vee}$ 
   does not belong to any boundary; see \cite{LeeSturmfels,BS18}.
   We expect that the same holds in general.

Furthermore,
    notice that parts (2) of Theorems~\ref{teoremaOdd} and \ref{casopari} are not sharp. 
    Indeed for low degrees, $d\le 8$, we know by \cite{BS18} that the formula
 \begin{equation}\label{formula-sharp}\partial_{\textrm{alg}}(\mathcal{R}_{d,k+i}) \subseteq \bigcup_{\rho} (\Delta_{\rho})^{\vee} \end{equation}
    holds with equality when we take 
    only partitions
    {with length $\rho\in\{k-i-1, k-i\}$} if $d$ is odd, and 
    with length $\rho\in\{k-i, k-i+1\}$ if $d$ is even.
    We expect that the same happens in general. 
        This also would imply that an algebraic boundary exists only between a region $\mathcal{R}_{d,r}$ and $\mathcal{R}_{d,r+1}$ (or $\mathcal{R}_{d,r-1}$).
\end{remark}

\begin{table}
\footnotesize 
\parbox{.45\linewidth}{
\tabcolsep=2.2pt 
\begin{tabular}{c|c|cccccccccccccc}
$d$ & {$k\diagdown i$} & 0& 1& 2& 3& 4& 5& 6& 7& 8& 9& 10& 11 \\
\hline 
5& 3&1&1& & & & & & & & & &  \\
7& 4&1&2&1& & & & & & & & &  \\
9& 5&1&3&3&1& & & & & & & &  \\
11& 6&1&3&5&4&1& & & & & & &  \\
13& 7&1&3&6&8&5&1& & & & & &  \\
15& 8&1&3&7&11&12&6&1& & & & &  \\
17& 9&1&3&7&13&18&16&7&1& & & &  \\
19& 10&1&3&7&14&23&27&21&8&1& & &  \\
21& 11&1&3&7&15&26&37&39&27&9&1& &  \\
23& 12&1&3&7&15&28&44&57&54&33&10&1&  \\
25& 13&1&3&7&15&29&49&71&84&72&40&11&1 \\
\end{tabular}
\caption{
\footnotesize Number of the partitions of $d=2k-1$ with all parts 
$\geq 2$
and 
of length  $k-i-1$ with $0\leq i\leq k-2$ (Theorem~\ref{teoremaOdd}).}
\label{tabella-1}
}
\parbox{0.5\linewidth}{
\centering 
\tabcolsep=2.2pt 
\begin{tabular}{c|c|ccccccccccccccc}
$d$ & {$k\diagdown i$} & 0 & 1& 2& 3& 4& 5& 6& 7& 8& 9& 10& 11& 12 \\
\hline 
6& 3&1&2&1& & & & & & & & & &  \\
8& 4&1&2&3&1& & & & & & & & &  \\
10& 5&1&2&4&4&1& & & & & & & &  \\
12& 6&1&2&5&7&5&1& & & & & & &  \\
14& 7&1&2&5&9&10&6&1& & & & & &  \\
16& 8&1&2&5&10&15&14&7&1& & & & &  \\
18& 9&1&2&5&11&18&23&19&8&1& & & &  \\
20& 10&1&2&5&11&20&30&34&24&9&1& & &  \\
22& 11&1&2&5&11&21&35&47&47&30&10&1& &  \\
24& 12&1&2&5&11&22&38&58&70&64&37&11&1&  \\
26& 13&1&2&5&11&22&40&65&90&101&84&44&12&1\\
\end{tabular}
\caption{\footnotesize Number of the partitions of $d=2k$ with all parts 
$\geq2$
and 
of length  $k-i$ with $0\leq i\leq k-1$ (Theorem~\ref{casopari}).}
\label{tabella-2}
}
\end{table}

In Tables~\ref{tabella-1} and \ref{tabella-2}, we report the number of partitions of given length of some integers,
in accordance to the formulas of Theorems~\ref{teoremaOdd} and \ref{casopari}.
The sum of the elements of the $k$-th row up to position $i$
gives an upper bound for the number of components 
of the algebraic boundary $\partial_{\textrm{alg}}(\mathcal{R}_{d,k+i})$, where $d=2k-1$ if $d$ is odd, and $d=2k$ if $d$ is even.
However, by taking into account our expectations described 
 in Remark \ref{quasi-finito-1},
a finer upper bound for the number 
of components of 
$\partial_{\textrm{alg}}(\mathcal{R}_{d,k+i})$ 
would be given just by the sum of the two numbers in the $k$-th row corresponding to the position $i$ and $i-1$.


\providecommand{\bysame}{\leavevmode\hbox to3em{\hrulefill}\thinspace}
\providecommand{\MR}{\relax\ifhmode\unskip\space\fi MR }
\providecommand{\MRhref}[2]{%
  \href{http://www.ams.org/mathscinet-getitem?mr=#1}{#2}
}
\providecommand{\href}[2]{#2}

\end{document}